\documentclass[11pt]{article}

\usepackage{bm} 
\usepackage{graphicx}
\usepackage{epstopdf}
\usepackage{float}
\usepackage{amsthm}
\usepackage[numbers]{natbib}
\usepackage{amsmath}
\usepackage{amsfonts}
\usepackage{color}
\usepackage{mathtools}

\newtheorem{theorem}{Theorem}

\newtheorem{proposition}{Proposition}
\newtheorem{corollary}{Corollary}

\newtheorem{lemma}{Lemma}

\usepackage{typearea}
\typearea{12}

\begin{document}
\title{Information Geometry of Wasserstein Statistics \\on Shapes and Affine Deformations}
\author{Shun-ichi Amari \\
	Teikyo University, Advanced Comprehensive  Research Organization, \\
	RIKEN Center for Brain Science \and
	Takeru Matsuda \\
	The University of Tokyo, \\
	RIKEN Center for Brain Science}
\date{}
\maketitle

\begin{abstract}
Information geometry and Wasserstein geometry are two main structures introduced in a manifold of probability distributions, and they capture its different characteristics. 
We study characteristics of Wasserstein geometry in the framework of \cite{LZ2019} for the affine deformation statistical model, which is a multi-dimensional generalization of the location-scale model.
We compare merits and demerits of estimators based on information geometry and Wasserstein geometry. 
The shape of a probability distribution and its affine deformation are separated in the Wasserstein geometry, showing its robustness against the waveform perturbation in exchange for the loss in Fisher efficiency. 
We show that the Wasserstein estimator is the moment estimator in the case of the elliptically symmetric affine deformation model.  
It coincides with the information-geometrical estimator (maximum-likelihood estimator) when the waveform is Gaussian.  
The role of the Wasserstein efficiency is elucidated in terms of robustness against waveform change.  
\end{abstract}

	\section{Introduction}
\label{sec:1}

We study a probability distribution $p({\bm{x}})$ over ${\bm{x}} \in X={\bm{R}}^d$, by using both information geometry (see \citet{Amari2016,Ay2017} etc.) and Wasserstein geometry (see \citet{Villani2003,PC2019,Santambrogio2015} etc.). When $d=2$, $p({\bm{x}})$ can be regarded as a visual pattern on ${\bm{R}}^2$.  

There are lots of applications of Wasserstein geometry to statistics (see e.g. \citet{AM2022,BJGR2019,yatracos2022,BBR2006,Matsuda2021, IOH2022, LM2020, Chen}), machine learning (see e.g. \citet{ACB2917,FZMAP2015,WL2019,PC2019,MMC2015}) and statistical physics (\citet{ito2023}).  
We also recommend the review paper and book by \citet{PZ2019, PZ2022}, which include lots of references.
However, applications to statistical inference look still premature.  

The affine deformation statistical model $p({\bm{x}}, {\bm{\theta}})$ is defined as \begin{equation}
	\label{eq:am220230131}
	p({\bm{x}}, {\bm{\theta}}) = |\Lambda|
	f \left( \Lambda ({\bm{x}}-{\bm{\mu}}) \right),
\end{equation}
where $f({\bm{z}})$ is a standard shape distribution satisfying
\begin{eqnarray}
	\label{eq:am2}
	\int f ({\bm{z}})d{\bm{z}} &=& 1, \\
	\label{eq:am3}
	\int {\bm{z}} f({\bm{z}})d{\bm{z}} &=& 0, \\
	\label{eq:am4}
	\int {\bm{z}}{\bm{z}}^{\top} f({\bm{z}}) d{\bm{z}} &=& I,
\end{eqnarray}              
and $I$ is the identity matrix.  
We also refer to the standard shape  $f$ as a ``waveform" in the following.
The deformation parameter consists of ${\bm{\theta}}=\left( {\bm{\mu}}, \Lambda \right) \in \Theta$ such that ${\bm{\mu}}$ is a vector specifying translation of the location and $\Lambda$ is a non-singular matrix representing scale changes and rotations of ${\bm{x}}$.
Given a standard $f$, we have a statistical model parameterized by ${\bm{\theta}}$: $\mathcal{M}_f = \left\{ p ({\bm{x}}, {\bm{\theta}}) \right\}$.
Geometrically, it forms a finite-dimensional statistical manifold, where ${\bm{\theta}}$ plays the role of a coordinate system.  
Note that it is not necessarily identifiable in general. The identifiability depends on $f$.
The deformation model is a generalization of the location-scale model. 
Note that this model is often called the location-scatter model in several fields such as statistics and signal processing \citep{tyler1987distribution,ollila2014regularized}.
Let $T_{\bm{\theta}}$ denote the affine deformation from ${\bm{x}}$ to ${\bm{z}}$ given by
\begin{equation*}
	{\bm{z}} = T_{\bm{\theta}}{\bm{x}} = \Lambda ({\bm{x}}-{\bm{\mu}}).
\end{equation*}

Let ${\mathcal{F}}=\{p({\bm{x}})\}$ be the space of all smooth positive probability density functions that have finite second moments.
Let ${\mathcal{F}}_S=\{f({\bm{z}})\}$ be its subspace consisting of all the standard distributions $f({\bm{z}})$ satisfying \eqref{eq:am3} and \eqref{eq:am4}.  
Then, any $q({\bm{x}}) \in {\mathcal{F}}$ is written in the form
\begin{equation*}
	q({\bm{x}}) = |\Lambda| f \left( \Lambda ({\bm{x}}-{\bm{\mu}}) \right)
\end{equation*}
for $f \in {\mathcal{F}}_S$ and ${\bm{\theta}}=( {\bm{\mu}}, \Lambda) \in \Theta$.
Note that $\bm{\theta}$ is not necessarily unique due to possible symmetries in $f$.
Hence, $\mathcal{F} = \mathcal{F}_S \times \Theta / \sim$, where $\sim$ is the equivalence relation of equality in distribution.
See Figure~\ref{fig1}. 

\begin{figure}
	\centering
	\includegraphics[width=8cm]{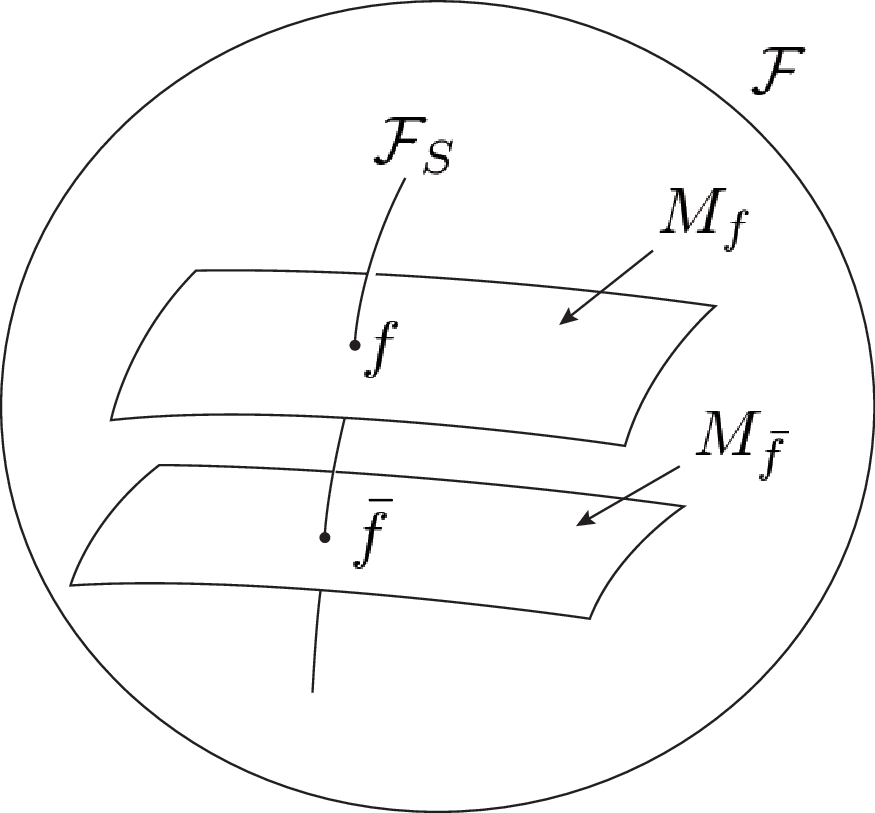}
	\caption{Decomposition of $\mathcal{F}$}
	\label{fig1}
\end{figure}

Geometry of a manifold of probability distributions has so far been studied by information geometry and Wasserstein geometry.  The two geometries capture different aspects of a manifold of probability distributions.  We use a divergence measure to explain this.  Let $D_F [p({\bm{x}}),q({\bm{x}})]$ and $D_W [p({\bm{x}}),q({\bm{x}})]$ be two divergence measures between distributions $p({\bm{x}})$ and $q({\bm{x}})$, where subscripts $F$ and $W$ represent Fisher-based information geometry and Wasserstein geometry, respectively.  Information geometry uses an invariant divergence $D_F$, typically the Kullback--Leibler divergence.  Wasserstein divergence $D_W$ is defined by the cost of transporting masses distributed in form $p({\bm{x}})$ to another $q({\bm{x}})$.  Roughly speaking, $D_F$ measures the vertical differences of $p({\bm{x}})$ and $q({\bm{x}})$, for example, represented by their log-ratio $\log (p({\bm{x}})/q({\bm{x}}))$, whereas $D_W$ measures the horizontal differences of $p({\bm{x}})$ and $q({\bm{x}})$ which corresponds to the transportation cost from $p({\bm{x}})$ to $q({\bm{x}})$. See Figure~\ref{fig3}.

\begin{figure}
	\centering
	\includegraphics[width=12cm]{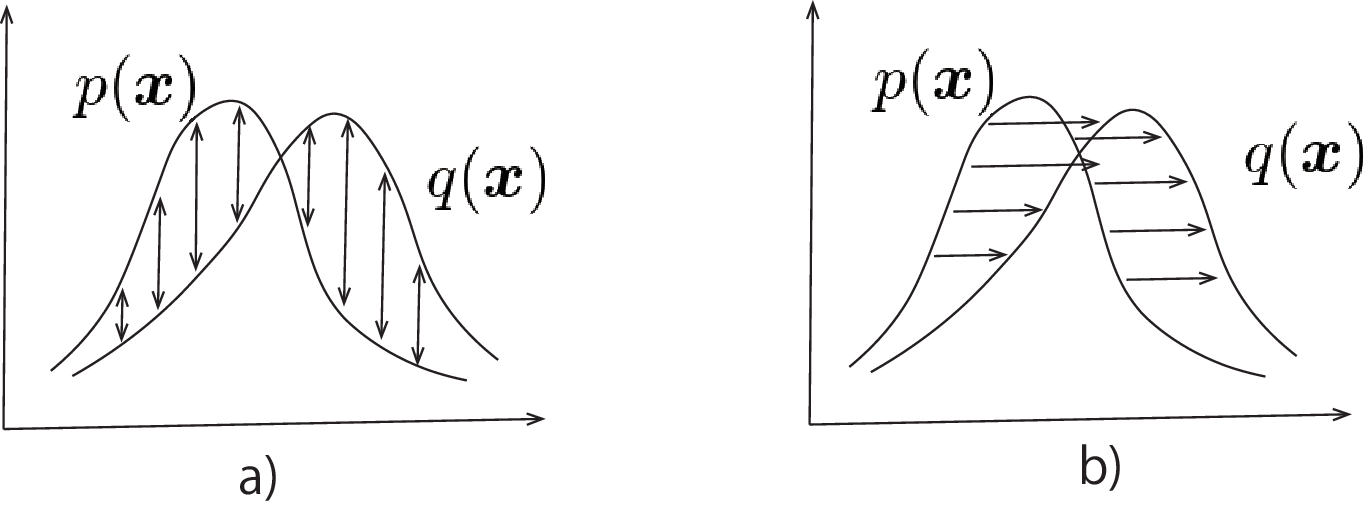}
	\caption{(a) $F$-divergence. (b) $W$-divergence.}
	\label{fig3}
\end{figure}

Information geometry is constructed based on the invariance principle of Chentsov \citep{Chentsov} such that $D_F [p({\bm{x}}),q({\bm{x}})]$ is invariant under invertible transformations of the coordinates ${\bm{x}}$ of the sample space $X$.  This implies that the divergence does not depend on the coordinate system of $X$.  We then have a unique Riemannian metric, which is the Fisher--Rao metric, and also a dual pair of affine connections \citep{AmariNagaoka2007}.  This is useful not only for analyzing the performances of statistical inference but also for image analysis, machine learning, statistical physics, and many others (see \citep{Amari2016}).

Wasserstein geometry has an old origin, proposed by G. Monge in 1781 as a problem of transporting mass distributed in the form $p({\bm{x}})$ to another $q({\bm{x}})$ such that the total transportation cost is minimized.  It depends on the transportation cost $c({\bm{x}}, {\bm{y}})$ between two locations ${\bm{x}}, {\bm{y}} \in X$.  The cost is usually a function of the Euclidean distance between ${\bm{x}}$ and ${\bm{y}}$.  We use the square of the distance as a cost function, which gives $L^2$-Wasserstein geometry.  This Wasserstein geometry directly depends on the Euclidean distance of $X={\bm{R}}^d$.  Therefore, it is useful for problems that intrinsically depend on the metric structure of $X$, such as the transportation problem, non-equilibrium statistical physics, pattern analysis, machine learning and many others.
However, it is in general difficult to calculate the Wasserstein distance \citep{PC2019}.
See recent papers \cite{JL2020,PC2019} for computational algorithms of the Wasserstein distance.

It is natural to search for the relation between the two geometries. There are a number of such trials, including \citet{AKO2018,AKOC2019,KZ2020,RW2023, ito2023,WY2022,Chizat,Kondratyev,Liero} and others.
See \citet{KZ2022} for a survey.

\citet{ChenLi} studied a general regular statistical model specified by a finite number of parameters and pulled-back a Riemannian metric from the Otto metric in the space consisting of all smooth and positive density functions on $\mathbb{R}^n$. 
Using this metric (W-metric), they studied the natural gradient estimation procedures, obtaining the result that the W-natural gradient has better performances compared to the natural gradient procedures due to the Fisher-Rao metric. 
This is remarkable observation of showing the usefulness of W-metric in statistical inference. 
We follow this idea of pulling-back the Otto Riemannian metric to a regular statistical model, using their regularity conditions which guarantee the existence and uniqueness of the pulled-back Riemannian metric.

\citet{LZ2019} gave a unified framework for the two geometries. The present article is based on their framework and focuses on the affine deformation model, for which the standard waveform $f$ and the deformation parameter ${\bm{\theta}}$ are separated. 
\citet{LZ2019} further introduced the Wasserstein score function in parallel to the Fisher score function, defining two estimators $\hat{\bm{\theta}}_F$  and $\hat{\bm{\theta}}_W$ thereby.  
The former $\hat{\bm{\theta}}_F$ is the maximum likelihood estimator that maximizes the log likelihood.  
This is the one that minimizes an invariant divergence from the empirical distribution $\hat{p}({\bm{x}})$ to a parametric model, where the empirical distribution is given based on $n$ independent observations ${\bm{x}}_1, \cdots, {\bm{x}}_n$ as
\begin{equation*}
	\hat{p}({\bm{x}}) = \frac 1n \sum_{t=1}^n \delta
	\left( {\bm{x}}-{\bm{x}}_t \right),
\end{equation*}
and $\delta({\bm{x}})$ is the delta function.  
The latter Wasserstein estimator $\hat{\bm{\theta}}_W$ is defined as the zero point of the Wasserstein score.
It is asymptotically equivalent to the minimizer of the $W$-divergence between the empirical distribution and model.
Also, \citet{LZ2019} further defined the $F$-efficiency and $W$-efficiency of a consistent estimator $\hat{\bm{\theta}}$ given a statistical model $\mathcal{M}=\left\{p({\bm{x}}, {\bm{\theta}}) \right\}$, proving the Cram\'{e}r-Rao type inequalities. 

The present paper is organized as follows. 
In Section \ref{sec:2}, we introduce two divergences between distributions, one based on the invariance principle and the other based on the transportation cost.  
The divergences give two Riemannian structures in the space ${\mathcal{F}}$ of probability distributions $p({\bm{x}})$ over $X={\bm{R}}^d$.  
A regular statistical model $\mathcal{M}= \left\{ p({\bm{x}}, {\bm{\theta}})\right\}$ parameterized by ${\bm{\theta}}$ is a finite-dimensional submanifold embedded in ${\mathcal{F}}$.  
In Section \ref{sec:3}, we define the $F$- and $W$-score functions following \cite{LZ2019}.  
The Riemannian structure of the tangent space of probability distributions is pulled-back to the model submanifold, giving both the Riemannian metrics and score functions.  
We define the $F$- and $W$-estimators $\hat{\bm{\theta}}_F$ and $\hat{\bm{\theta}}_W$ by using the $F$- and $W$-score functions, respectively.
Section \ref{sec:4} defines the affine deformation statistical model.  
Section \ref{sec:ellip} studies the elliptically symmetric affine deformation model $\mathcal{M}_f$, where $f$ is a spherically symmetric standard form.
For this model, we show that the $W$-score functions are quadratic functions of ${\bm{x}}$.  
Hence, it is proved that $\hat{\bm{\theta}}_W$ is a moment estimator.   
For the Gaussian shape, the $F$-estimator and $W$-estimator coincide.
We also show that $\mathcal{M}_f$ and ${\mathcal{F}}_S$ are orthogonal in the $W$-geometry, implying the separation of the waveform and deformation.
In Section \ref{sec:robust}, we elucidate the role of $W$-efficiency from the point of view of robustness to a change in the waveform $f$ due to observation noise. 
Section \ref{sec:conc} briefly summarizes the paper and mentions future work.

\section{Riemannian structures in the space of probability densities}\label{sec:2}

We consider the space ${\mathcal{F}}= \left\{p({\bm{x}}) \right\}$ of all smooth positive probability density functions on ${\bm{R}}^d$ that have finite second moments.  
Later, we may relax the conditions of positivity and smoothness, when we discuss a parametric model, in particular the deformation model\footnote{For example, the probability density of the uniform distribution inside an ellipse takes zero outside of the ellipse and thus non-smooth on the boundary.}.  
We define a divergence function $D[p({\bm{x}}), q({\bm{x}})]$, which represents the degree of difference between $p({\bm{x}})$ and $q({\bm{x}})$.  The square of the $L^2$ distance between $p({\bm{x}})$ and $q({\bm{x}})$ plays this role, but a divergence does not necessarily need to be symmetric with respect to $p({\bm{x}})$ and $q({\bm{x}})$.  A divergence function satisfies the following conditions:
\begin{enumerate}
	\item $D[p({\bm{x}}),q({\bm{x}})] \ge 0$ and the equality holds if and only if $p({\bm{x}}) = q({\bm{x}})$.
	\item  Let $\delta p({\bm{x}})$ be an infinitesimally small deviation of $p({\bm{x}})$.  Then, $D[p({\bm{x}}),p({\bm{x}}) +  \delta p({\bm{x}})]$ is approximated by a positive quadratic functional of $\delta p({\bm{x}})$.
\end{enumerate}

A divergence is said to be invariant if
\begin{equation*}
	D [ p({\bm{x}}),q({\bm{x}}) ] = 
	D [ \tilde{p}({\bm{y}}), \tilde{q}({\bm{y}}) ]
\end{equation*}
holds for every smooth reversible transformation ${\bm{k}}$ of the coordinates from ${\bm{x}} \in {\bm{R}}^d$ to ${\bm{y}}={\bm{k}}({\bm{x}})$, where 
\begin{equation*}
	\tilde{p} ({\bm{y}}) =
	\left|
	\frac{\partial {\bm{x}}}{\partial {\bm{y}}}
	\right| p({\bm{x}}).
\end{equation*}
A typical invariant divergence is the $\alpha$-divergence $(\alpha \ne \pm 1)$ defined by
\begin{eqnarray*}
	D_{\alpha}[ p({\bm{x}}),q({\bm{x}}) ] &=& \frac 4{1-\alpha^2} \left( 1-\int p({\bm{x}})^{(1+\alpha)/2} q({\bm{x}})^{(1-\alpha)/2}  {\rm d} \bm{x} \right)
\end{eqnarray*} 
for $\alpha \ne \pm 1$.
For $\alpha= 1$, we define $D_1[p,q]$ by the Kullback--Leibler divergence
\begin{equation*}
	D_1 [ p({\bm{x}}),q({\bm{x}}) ] = \int p({\bm{x}}) \log \frac{p({\bm{x}})}{q({\bm{x}})} {\rm d} \bm{x}.
\end{equation*}
For $\alpha=-1$, we define $D_{-1}[ p({\bm{x}}),q({\bm{x}}) ]=D_1[ q({\bm{x}}),p({\bm{x}}) ]$.
The case $\alpha=0$ is equivalent to the Hellinger divergence
\begin{equation*}
	H^2 [ p({\bm{x}}),q({\bm{x}}) ] = \frac{1}{2} \int \left(
	\sqrt{p({\bm{x}})} - \sqrt{q({\bm{x}})} \right)^2 {\rm d} \bm{x}.
\end{equation*}
A characterization of the $\alpha$-divergence is given in \cite{Amari2016}.  
The $\alpha$-divergence gives information-geometric structure to $\mathcal{F}$.

Another divergence is the Wasserstein divergence.  
Let us transport masses piled in the form $p({\bm{x}})$ to another $q({\bm{x}})$.  
To this end, we need to move some mass at ${\bm{x}}$ to another position ${\bm{y}}$.  
Let $\pi ({\bm{x}}, {\bm{y}})$ be a coupling of $p({\bm{x}})$ and $q({\bm{y}})$:
\begin{eqnarray}
	\label{eq:am1420230116}
	\int \pi ({\bm{x}}, {\bm{y}}) {\rm d} {\bm{y}}
	&=& p({\bm{x}}), \\
	\label{eq:am1520230116}
	\int \pi ({\bm{x}}, {\bm{y}}) {\rm d} \bm{x}
	&=& q({\bm{y}}).
\end{eqnarray}
Note that $\pi$ does not necessarily have a density.
For convenience, we use the notation $\pi ({\bm{x}}, {\bm{y}})$ in this paper.
Let $c({\bm{x}}, {\bm{y}})$ be the cost of transporting a unit of mass from ${\bm{x}}$ to ${\bm{y}}$.
Then, the Wasserstein divergence $D_W [p({\bm{x}}),q({\bm{x}})]$ is the minimum transporting cost from $p({\bm{x}})$ to $q({\bm{x}})$.  By using stochastic plan $\pi ({\bm{x}}, {\bm{y}})$,  the Wasserstein divergence between $p({\bm{x}})$ and $q({\bm{x}})$ is given by
\begin{equation*}
	D_W \left[p({\bm{x}}),q({\bm{x}}) \right] =
	{\mathop{\inf}_{\pi}} \int c({\bm{x}}, {\bm{y}}) \pi
	({\bm{x}}, {\bm{y}}) {\rm d} \bm{x} {\rm d}{\bm{y}},
\end{equation*}
where infimum is taken over all stochastic plans $\pi$ satisfying (\ref{eq:am1420230116}) and (\ref{eq:am1520230116}).  
When the cost is the square of the Euclidean distance
\begin{equation*}
	c({\bm{x}}, {\bm{y}}) = \| {\bm{x}}-{\bm{y}} \|^2,
\end{equation*}
we call $D_W$ the $L^2$-Wasserstein divergence. 
We focus on this divergence in the following.
Note that the $L^2$-Wasserstein divergence is the square of the $L^2$-Wasserstein distance.
From Brenier’s theorem \cite{Brenier87,Brenier91}, the optimal transport is actually induced by a transport map. 
In other words, for each point $\bm{x}$, $\pi(\bm{x}, \cdot)$ is supported at a single point.

The dynamic formulation of the optimal transport problem proposed by \cite{Brenier} and developed further by \cite{BB2000,McCann} is useful. Let $\rho ({\bm{x}}, t)$ be a family of probability distributions parameterized by $t$.  It represents the time course $\rho({\bm{x}}, t)$ of transporting $p({\bm{x}})$ to $q({\bm{x}})$, satisfying 
\begin{equation*}
	\rho({\bm{x}}, 0) = p({\bm{x}}), \quad
	\rho({\bm{x}}, 1) = q({\bm{x}}).
\end{equation*}
We introduce potential $\Phi ({\bm{x}}, t)$ such that its gradient $\nabla_{\bm{x}} \Phi ({\bm{x}}, t)$ represents the velocity 
\begin{equation*}
	{\bm{v}}({\bm{x}}, t) =
	\nabla_{\bm{x}} \Phi({\bm{x}}, t)
\end{equation*}
of mass flow at ${\bm{x}}$ and $t$ in the dynamic plan. 
Then, $\Phi({\bm{x}}, t)$ satisfies the following continuity equation
\begin{equation}
	\label{eq:am1920230118}
	\partial_t \rho ({\bm{x}}, t) + \nabla_{\bm{x}} \cdot
	\left\{ \rho ({\bm{x}}, t) \nabla_{\bm{x}} \Phi ({\bm{x}}, t) \right\}
	= 0.
\end{equation}
The Wasserstein divergence is written in the dynamic formulation as
\begin{equation*}
	D_W \left[ p({\bm{x}}),q({\bm{x}}) \right] =
	{\mathop{\inf}_{\Phi}} \int^1_0 \int 
	\| \nabla_{\bm{x}} \Phi ({\bm{x}}, t) \|^2 \rho({\bm{x}}, t) {\rm d} \bm{x} {\rm d} t.
\end{equation*}

We introduce a Riemannian structure to $\mathcal{F}$ by the Taylor expansion of $D[p,p + \delta p]$.  
The Riemannian metric $g$ gives the squared magnitude ${\rm d} s^2$ of an infinitesimal deviation $\delta p ({\bm{x}})$ in the tangent space of $\mathcal{F}$, for example, by
\begin{equation*}
	{\rm d} s^2 = \langle \delta p({\bm{x}}), g (\delta p({\bm{x}})) \rangle, 
\end{equation*}
where
\begin{equation*}
	\langle
	a({\bm{x}}), \bm{b}({\bm{x}})
	\rangle = \int a({\bm{x}}) \bm{b}({\bm{x}}) {\rm d} \bm{x}.
\end{equation*}
In the case of the invariant divergence, 
\begin{equation}
	g_F (\delta p(\bm{x}))= \frac{\delta p(\bm{x})}{p(\bm{x})}, \label{gF}
\end{equation}
so that
\begin{equation*}
	{\rm d} s^2 = \int \frac{\delta p(\bm{x})^2}{p(\bm{x})} {\rm d} \bm{x}.
\end{equation*}
In the case of the $L^2$-Wasserstein divergence, consider the change of density from $\rho ({\bm{x}}, 0)=p({\bm{x}})$ at $t=0$ to $\rho ({\bm{x}}, {\rm d} t) = p({\bm{x}})+ \delta p({\bm{x}})$ at $t={\rm d} t$ for an infinitesimal ${\rm d} t$.
By using the potential $\Phi({\bm{x}})$ of this infinitesimal transport,
\begin{equation*}
	\delta p({\bm{x}}) = -\Delta_p \Phi({\bm{x}}) {\rm d} t + o({\rm d} t),
\end{equation*}
where $\Delta_p$ is the operator defined by
\begin{equation*}
	\Delta_p \Phi ({\bm{x}}) = \nabla_{\bm{x}} \cdot
	\left( p({\bm{x}}) \nabla_{\bm{x}} \Phi ({\bm{x}}) \right).
\end{equation*}
Then, the $L^2$-Wasserstein divergence is 
\begin{align*}
	D_W [p({\bm{x}}),p({\bm{x}})+ \delta p({\bm{x}})] &= \int \| \nabla_{\bm{x}} \Phi ({\bm{x}}) {\rm d} t \|^2 p({\bm{x}}) {\rm d} \bm{x} + o({\rm d} t^2) \\
	&= \langle \Phi({\bm{x}}) {\rm d} t, -\Delta_{p} \Phi ({\bm{x}}) {\rm d} t \rangle + o({\rm d} t^2) \\
	&= \langle \Phi({\bm{x}}) {\rm d} t, \delta p ({\bm{x}}) \rangle+ o({\rm d} t^2),
\end{align*}
where we used integration by parts
\begin{equation}
	\int (\nabla_{\bm{x}} a({\bm{x}}))^{\top} (\nabla_{\bm{x}} \bm{b}({\bm{x}}) ) p({\bm{x}}) {\rm d} \bm{x} = -\int a({\bm{x}}) \Delta_p \bm{b}({\bm{x}}) {\rm d} \bm{x}, 	\label{eq:am3520210118}
\end{equation}
for $a(\bm{x})$ decaying sufficiently fast.
See \cite{ChenLi} for precise regulaity conditions.
Thus, 
\begin{equation*}
	g_W (\delta p(\bm{x})) = \Phi (\bm{x}) {\rm d} t.
\end{equation*}
Note that $\Phi (\bm{x})$ is unique up to an additive constant under regularity conditions (see Theorem 13.8 of \cite{Villani2009} and Section 8.1 of \cite{Villani2003}).
This is Otto's Riemannian metric \cite{Otto}.

We focused on the space $\mathcal{F}$ of smooth and positive densities under $L^2$-Wasserstein geometry.
It is indeed an infinite-dimensional Riemannian manifold \cite{Lott}.
Since the space $\mathcal{F}_S$ is a co-dimension $d+d(d+1)/2$ subspace of $\mathcal{F}$ that specifies the value of linear functionals (first and second moments), it is also an infinite-dimensional Riemannian manifold. 
However, if we consider the larger space of general probability distributions, then it is not even a Banach manifold under $L^2$-Wasserstein geometry, because it includes singular (atomic) distributions for which the tangent space is more restricted than that of $\mathcal{F}$.
Note that information geometry and Wasserstein geometry induce different topologies on the space of probability distributions.
However, this does not bother us when we study only a finite-dimensional regular statistical model included in $\mathcal{F}$.
We follow the regularity conditions given in \citet{ChenLi} throughout the paper.

\section{Score functions and estimators} \label{sec:3}

We consider a regular statistical model $\mathcal{M}=\left\{ p({\bm{x}}, {\bm{\theta}}) \right\} \subset \mathcal{F}$ parameterized by an $m$-dimensional vector ${\bm{\theta}}$.  
The tangent space $T_{\theta} \mathcal{M}$ of $\mathcal{M}$ at ${\bm{\theta}}$ is spanned by $m$ functions
\begin{equation*}
	\partial_i p ({\bm{x}}, {\bm{\theta}}) =
	\frac{\partial}{\partial \theta^i}
	p ({\bm{x}}, {\bm{\theta}})
\end{equation*}
for $i=1, \dots, m$, so that a tangent vector $\delta p({\bm{x}})$ is given by
\begin{equation}
	\label{eq:am2920230131}
	\delta p({\bm{x}}) = \partial_i p ({\bm{x}}, {\bm{\theta}})
	d \theta^i.
\end{equation}
Hereafter, the summation convention is used, that is, all indices appearing twice, once as upper and the other as lower indices, e.g. $i$'s in (\ref{eq:am2920230131}), are summed up.

Let us define the score functions $S_i ({\bm{x}}, {\bm{\theta}})$ by using the basis functions $\partial_i p ({\bm{x}}, {\bm{\theta}})$ of the tangent space of $\mathcal{M}$ for $i=1, \cdots, m$. 
In the case of the invariant Fisher geometry, from \eqref{gF}, the Fisher score function $S^F_i ({\bm{x}}, {\bm{\theta}})$ is defined by
\begin{equation*}
	S^F_i ({\bm{x}}, {\bm{\theta}}) =
	\frac{\partial_i p ({\bm{x}}, {\bm{\theta}})}{p({\bm{x}}, {\bm{\theta}})}
	= \partial_i l ({\bm{x}}, {\bm{\theta}}), \quad l ({\bm{x}}, {\bm{\theta}}) = \log p ({\bm{x}}, {\bm{\theta}}),
\end{equation*}
which is the derivative of log-likelihood. 
In Wasserstein geometry, the Wasserstein score ($W$-score) function $S^W_i ({\bm{x}}, {\bm{\theta}})$ \cite{LZ2019} is defined as the solution of
\begin{equation*}
	\Delta_p S^W_i ({\bm{x}}, {\bm{\theta}}) = -\partial_i p ({\bm{x}}, {\bm{\theta}}), \quad {\rm{E}}_{\theta} [ S_i^W ({\bm{x}}, {\bm{\theta}})] = 0,
\end{equation*} 
where the latter condition is imposed to eliminate the indefiniteness due to the integral constant.
By using \eqref{eq:am3520210118} provided it holds, we see that $S^W_i ({\bm{x}}, {\bm{\theta}})$ satisfies the Poisson equation:
\begin{equation}
	\label{eq:poisson}
	\nabla_{\bm{x}} \log p ({\bm{x}}, {\bm{\theta}}) \cdot \nabla_{\bm{x}} S^W_i ({\bm{x}}, {\bm{\theta}}) +
	\Delta_{\bm{x}} S^W_i
	({\bm{x}}, {\bm{\theta}}) +
	\frac{\partial}{\partial \theta_i} \log
	p({\bm{x}}, {\bm{\theta}}) = 0.
\end{equation}
For infinitesimal $\delta$, the map $\bm{x} \mapsto \bm{x}+\delta \nabla_{\bm{x}} S_i^{{W}}(\bm{x}, \bm{\theta})$ is the optimal transport map from $p(\bm{x}, \bm{\theta})$ to $p(\bm{x}, \bm{\theta} + \delta \bm{e}_i)$ with transportation cost
\begin{equation}
	D_W(p(\bm{x}, \bm{\theta}), p(\bm{x}, \bm{\theta} + \delta \bm{e}_i)) = \int \| \delta \nabla_{\bm{x}} S_i^{{W}}(\bm{x}, \bm{\theta}) \|^2 p(\bm{x}, \bm{\theta}) {\rm d} x + o(\delta^2)
\end{equation}
as $\delta \to 0$, where $\bm{e}_i$ is the $i$-th standard unit vector.
See Proposition 8.4.6 of \cite{Ambrosio} for more rigorous statement.
In both Fisher and Wasserstein cases, the score function satisfies
\begin{equation}
	\label{eq:am36}
	{\rm{E}}_{\theta} [ S_i ({\bm{x}}, {\bm{\theta}}) ] = 0.
\end{equation}
We think that the existence and uniqueness of $S_i^W$ can directly be established along the lines of Brenier's theorem, as suggested by one of the reviewers.
Note that \cite{ChenLi} discussed some results restricting the Wasserstein metric to a regular parameteric statitical manifold.

The Riemannian metric tensor $g_{ij}({\bm{\theta}})$ is pulled-back from $g$ in ${\mathcal{F}}$ to $\mathcal{M}$. 
Note that \citet{ChenLi} provided a detailed account on the restriction of the Wasserstein metric to a parametric statistical manifold with regularity conditions and concrete examples.
In the Fisherian case, 
\begin{equation*}
	g^F_{ij} ({\bm{\theta}}) = {\rm{E}}
	\left[ \partial_i l (\bm{x}, {\bm{\theta}}) \partial_j l (\bm{x}, {\bm{\theta}}) \right] =
	\int p (\bm{x}, {\bm{\theta}}) \partial_i l 
	({\bm{x}}, {\bm{\theta}}) \partial_j l
	({\bm{x}}, {\bm{\theta}}) {\rm d} \bm{x}.
\end{equation*}
In the Wasserstein case, 
\begin{align}
	g^W_{ij}({\bm{\theta}}) &= 	-\int S^W_i ({\bm{x}}, {\bm{\theta}}) \Delta_p S^W_j ({\bm{x}}, {\bm{\theta}}) {\rm d} \bm{x} \nonumber \\
	&=	\int p ({\bm{x}}, {\bm{\theta}})
	\nabla_{\bm{x}} S^W_i ({\bm{x}}, {\bm{\theta}})^{\top}
	\nabla_{\bm{x}} S^W_j ({\bm{x}}, {\bm{\theta}}) {\rm d} \bm{x} \nonumber \\
	&= {\rm E} [\nabla_{\bm{x}} S^W_i ({\bm{x}}, {\bm{\theta}})^{\top}
	\nabla_{\bm{x}} S^W_j ({\bm{x}}, {\bm{\theta}})], \label{WIM}
\end{align}
where identity (\ref{eq:am3520210118}) is used.

The score functions $S_i ({\bm{x}}, {\bm{\theta}})$  give a set of estimating functions from (\ref{eq:am36}), which are used to obtain an estimator $\hat{\bm{\theta}}$.  
Suppose that we have $n$ independent observations ${\bm{x}}_1, \cdots, {\bm{x}}_n$ from $p (\bm{x}, {\bm{\theta}})$.
Let $\hat{p}_{\rm{emp}}({\bm{x}})$ be the empirical distribution given by
\begin{equation*}
	\hat{p}_{\rm{emp}}({\bm{x}}) = \frac 1n
	\sum^n_{t=1} \delta
	\left({\bm{x}}-{\bm{x}}_t \right).
\end{equation*}
Then, replacing expectation ${\rm{E}}$ in (\ref{eq:am36}) by the expectation with respect to the empirical distribution, we have estimating equations
\begin{equation}
	\label{eq:am4420230516}
	{\rm{E}}_{\rm{emp}} 
	[ S_i ({\bm{x}}, \hat{\bm{\theta}}) ] =
	\frac 1n \sum^n_{t=1} S_i ( {\bm{x}}_t, \hat{\bm{\theta}} ) = 0, \quad i=1, \cdots, m.
\end{equation}
The solution $\hat{\bm{\theta}}$ gives a consistent estimator for large $n$ under regularity conditions on the score function \cite{vV}.  
Roughly speaking, $\hat{\bm{\theta}}$ is the projection of $\hat{p}_{\rm{emp}}({\bm{x}})$ to the model $\mathcal{M}$ with respect to the metric $g$ (see Figure~\ref{fig4}).  
It is the solution of
\begin{equation*}
	\langle \hat{p}_{\rm{emp}} ({\bm{x}}), S_i ({\bm{x}}, {\bm{\theta}}) \rangle
	= 0,
\end{equation*}
giving a consistent estimator $\hat{\bm{\theta}}$. 
A consistent estimator is Fisher efficient when the projection is orthogonal with respect to the Fisher--Rao metric \citep{AmariNagaoka2007}.

\begin{figure}
	\centering
	\includegraphics[width=8cm]{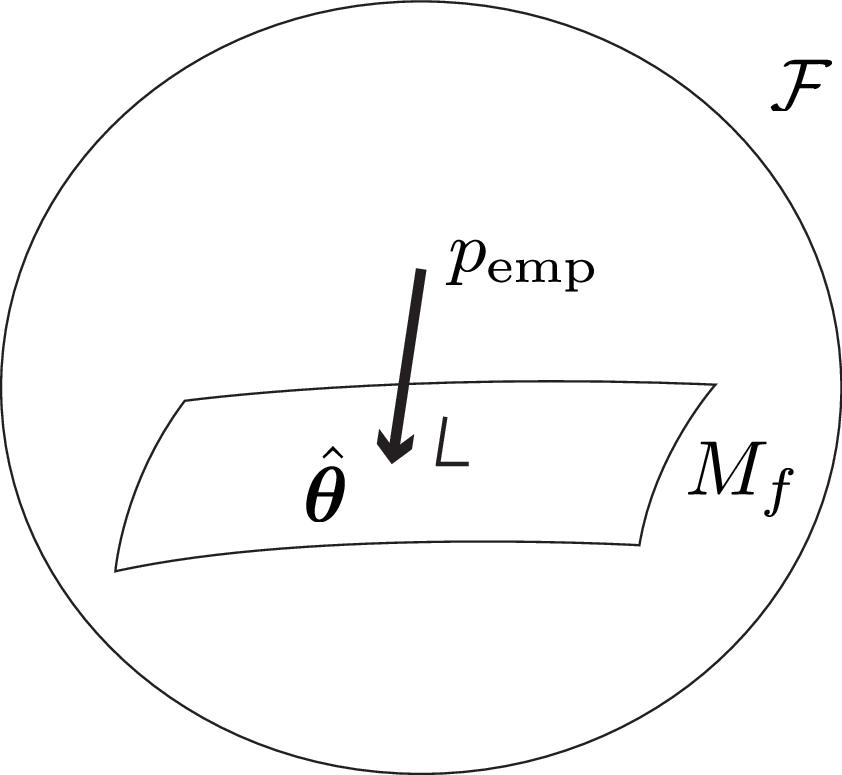}
	\caption{Projection of $\hat{p}_{\rm{emp}}$ to $\mathcal{M}$.}
	\label{fig4}
\end{figure}

For the invariant Fisherian case, the estimator $\hat{\bm{\theta}}_F$ defined by \eqref{eq:am4420230516} is the maximum likelihood estimator:
\begin{equation*}
	\frac 1n \sum^n_{t=1} S_i^F ( {\bm{x}}_t, \hat{\bm{\theta}}_F ) = 0.
\end{equation*}
The Cram\'{e}r-Rao theorem gives a matrix inequality for any unbiased estimator $\hat{\bm{\theta}}$,
\begin{equation*}
	{\rm{Cov}} [ \hat{\bm{\theta}} ] \succeq
	\frac 1n g^{-1}_F ({\bm{\theta}}),
\end{equation*}
where ${\rm{Cov}} [\cdot]$ is the covariance matrix and $\succeq$ denotes the matrix order defined by the positive semidefiniteness.  
The maximum likelihood estimator $\hat{\bm{\theta}}_F$ satisfies 
\begin{equation*}
	{\rm{Cov}} 
	[ \hat{\bm{\theta}}_F ] \approx
	\frac 1n g^{-1}_F ({\bm{\theta}})
\end{equation*}
asymptotically.  Hence, it minimizes the error covariance matrix and the minimized error covariance is given asymptotically by the inverse of the Fisher metric tensor $g_F$ divided by $n$.
Such a property is called the Fisher efficiency.

In the following, we study the characteristics of the estimator $\hat{\bm{\theta}}_{W}$ defined by \eqref{eq:am4420230516} with the Wasserstein score:
\begin{equation*}
	\frac 1n \sum^n_{t=1} S_i^W ( {\bm{x}}_t, \hat{\bm{\theta}}_W ) = 0.
\end{equation*}
We call $\hat{\bm{\theta}}_{W}$ the Wasserstein estimator ($W$-estimator) following \cite{LZ2019}.
In the case of the one-dimensional location-scale model, the Wasserstein estimator is asymptotically equivalent to the estimator obtained by minimizing the Wasserstein divergence (transportation cost) from the empirical distribution $\hat{p}_{\rm{emp}}({\bm{x}})$ to model $\mathcal{M}$:
\begin{equation}\label{am_est}
	\hat{\bm{\theta}}_{Wp} = {\mathop{\arg\min}_{{\bm{\theta}}}} \; D_W
	\left[ \hat{p}_{\rm{emp}}({\bm{x}}), p ({\bm{x}}, {\bm{\theta}}) \right].
\end{equation}
See the end of Section~\ref{sec:ellip}.
The properties of $\hat{\bm{\theta}}_{Wp}$ were studied in detail by \cite{AM2022} in the case of the one-dimensional location-scale model. 
Note that $\hat{\bm{\theta}}_{W} \neq \hat{\bm{\theta}}_{Wp}$ in general, contrary to the Fisher case.

\section{Affine deformation model}
\label{sec:4}
Now, we focus on the affine deformation model.
Let $f({\bm{z}}) \in \mathcal{F}_S$ be a standard probability density function satisfying (\ref{eq:am2}), (\ref{eq:am3}), and (\ref{eq:am4}).  
To define $\mathcal{M}_f$, we use affine deformation of ${\bm{x}}$ to ${\bm{z}}$ by
\begin{equation*}
	{\bm{z}} = \Lambda ({\bm{x}}-{\bm{\mu}}),
\end{equation*}
where ${\bm{\mu}}$ is a vector representing shift of location and $\Lambda$ is a non-singular matrix.  
Hence, ${\bm{\theta}}= ({\bm{\mu}}, \Lambda)$ is $m$-dimensional where $m \leq d+d^2$ due to the possible symmetries in $f$. 
The model $\mathcal{M}_f$ is defined from
\begin{equation*}
	p ({\bm{x}}, {\bm{\theta}}) {\rm d} \bm{x} =
	f({\bm{z}}) d{\bm{z}},
\end{equation*}
that is
\begin{equation*}
	p ({\bm{x}}, {\bm{\theta}}) = |\Lambda| f (\Lambda ({\bm{x}}-{\bm{\mu}})),
\end{equation*}
satisfying
\begin{equation} \label{moment}
	\int {\bm{x}} p ({\bm{x}}, {\bm{\theta}})
	{\rm d} \bm{x} = {\bm{\mu}}, 
	\quad \int (\bm{x}-\bm{\mu}) (\bm{x}-\bm{\mu})^{\top} p ({\bm{x}}, {\bm{\theta}}) {\rm d} \bm{x} = \Lambda^{-1} (\Lambda^{-1})^{\top}. 
\end{equation}
This is a generalization of the location-scale model, which is simply obtained by putting $\Lambda = (1/\sigma) I$, with $\sigma$ being the scale factor.  It should be noted that $\Lambda$ is decomposed as $\Lambda = U DO$, where $U$ and $O$ are orthogonal matrices and $D$ is a positive diagonal matrix (singular value decomposition).  
In the following, we denote the log probability of standard shape $f$ by 
\begin{equation*}
	l ({\bm{z}}) = \log f ({\bm{z}}).  
\end{equation*}

For each standard shape function $f \in {\mathcal{F}}_S$, an affine deformation model $\mathcal{M}_f$  parameterized by ${\bm{\theta}} = ({\bm{\mu}}, \Lambda) \in \Theta$ is attached.  Thus, ${\mathcal{F}}$ is decomposed as 
\begin{equation*}
	{\mathcal{F}} = {\mathcal{F}}_S \times \Theta / \sim,
\end{equation*}
where $\sim$ is the equivalence relation of equality in distribution.
For a general $f$, $\mathcal{M}_f$ has cone structure parameterized by $({\bm{\mu}}, D, U, O)$, where $\Lambda = U D O$ and $D$ is a diagonal matrix with diagonal elements $d_i>0$.
Thus, $D$ can be identified with a vector in the open positive quadrant $\bm{R}_+^d$ of ${\bm{R}}^d$, which has the cone structure. 
Since ${\bm{\mu}} \in {\bm{R}}^d$, and $U,O \in {\mathcal{O}}(d)$, we have the decomposition
\begin{equation*}
	\mathcal{M}_f = {\bm{R}}^d \times \bm{R}_+^{d} \times {\mathcal{O}}(d) \times {\mathcal{O}}(d) / \sim.
\end{equation*}
See \cite{Takatsu2012} for the cone structure of ${\mathcal{F}}$.  When $f$ is Gaussian, its structure is studied in detail by \cite{Takatsu2011}.

When $p({\bm{x}})$ belongs to $\mathcal{M}_f$, the waveform of $p({\bm{x}})$ is said to be equivalent to that of $f$.  
The space $\mathcal{M}_f$ consists of the distributions of all equivalent waveforms.
All ellipsoidal shapes are equivalent to a spherical shape.  A family of special parallel-piped shapes are equivalent to a cubic form (see Figure~\ref{fig5}).  Therefore, our model is useful for separating the effect of the shape from location and affine deformation.  

\begin{figure}
	\centering
	\includegraphics[width=12cm]{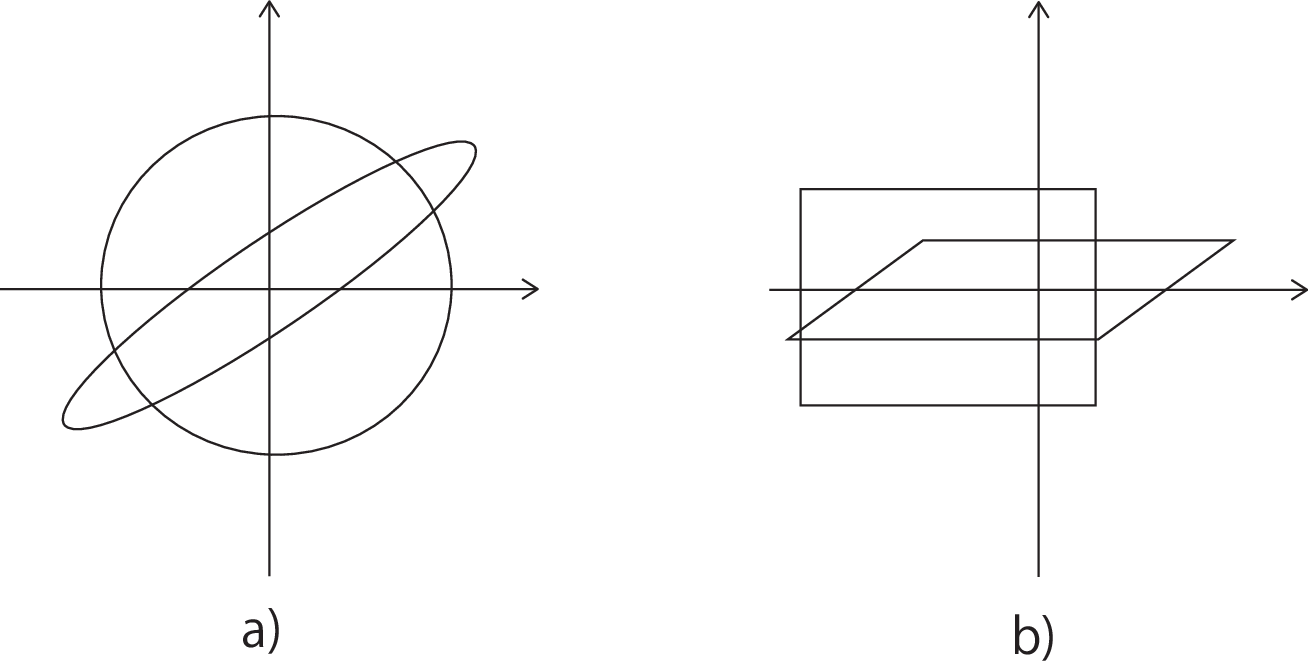}
	\caption{Equivalent shapes.}
	\label{fig5}
\end{figure}

We may consider subclasses of the transformation model.
One simple example is the location model, in which $\Lambda$ is fixed to the identity matrix $I$. A stronger theorem is known in such a simple model \citep{Givens}.
In our context, it can be expressed as follows.

\begin{proposition}\upshape
	Wasserstein geometry gives an orthogonal decomposition of the shape and locations,
	\begin{equation*}
		D_W \left[ 
		f_1 \left({\bm{x}}-{\bm{\mu}}_1 \right),  f_2 \left({\bm{x}}-{\bm{\mu}}_2 \right)
		\right] = D_W \left[f_1({\bm{x}}), f_2 ({\bm{x}}) \right] + \|{\bm{\mu}}_1-{\bm{\mu}}_2 \|^2.
	\end{equation*}
\end{proposition}

\section{Elliptically symmetric deformation model}\label{sec:ellip}
Here, we focus on the cases where $f(\bm{z})$ is spherically symmetric and thus written as $f(\bm{z})=g(\|\bm{z}\|)$ for some function $g$.
Thus,
\begin{equation}
	\label{eq:am61}
	p ({\bm{x}}, {\bm{\theta}}) =
	|\Lambda| g(\| \Lambda ({\bm{x}}-{\bm{\mu}}) \|),
\end{equation}
which is elliptically symmetric.
We restrict the parameter $\Lambda$ to be symmetric and positive definite in this section.
Namely, the dimension of $\mathcal{M}_f$ is $d+d(d+1)/2$.

First, we consider the $F$-estimator $\hat{\bm{\theta}}_F$ (maximum likelihood estimator).
The log-likelihood is given by
\begin{equation*}
	\log p ({\bm{x}}, {\bm{\theta}})  
	= \log |\Lambda| + \log g (\| \Lambda ({\bm{x}}- {\bm{\mu}}) \|).
\end{equation*}
When there are $n$ independent observations ${\bm{x}}_1, \cdots, {\bm{x}}_n$, summation is taken over them so that we have the likelihood equations
\begin{equation*}
	\sum^n_{j=1} \partial_{\bm{\theta}} \log p \left({\bm{x}}_j, {\bm{\theta}} \right) = 0.
\end{equation*}
The solution $\hat{\bm{\theta}}_F$ strongly depends of the shape $g$.

Contrary to this, the $W$-estimator $\hat{\bm{\theta}}_W$ does not depend on the shape $g$ as follows.
We write the $i$-th standard unit vector by $\bm{e}_i \in \mathbb{R}^d$ so that $(\bm{e}_i)_j=\delta_{ij}$.

\begin{lemma}\label{lem:quad}
	For a symmetric $A \in \mathbb{R}^{d \times d}$, $\bm{b} \in \mathbb{R}^d$ and $c \in \mathbb{R}$,
	\begin{align*}
		\nabla_{\bm{x}} \left( \frac{1}{2} \bm{x}^{\top} A \bm{x}+ \bm{b}^{\top} \bm{x} + c \right) = A \bm{x} + \bm{b}, \quad \Delta_{\bm{x}} \left( \frac{1}{2} \bm{x}^{\top} A \bm{x} + \bm{b}^{\top} \bm{x} + c \right) = {\rm tr} (A).
	\end{align*}
\end{lemma}
\begin{proof}
	Straightforward calculation.
\end{proof}

\begin{lemma}\label{lem:sylvester}
	Let $A,B \in \mathbb{R}^{d \times d}$ be symmetric matrices.
	If $A$ is positive definite, then the Sylvester equation $AX + XA = B$ has the unique solution $X$, which satisfies $X^{\top}=X$ and ${\rm tr} (X) = {\rm tr} (A^{-1} B)/2$.
\end{lemma}
\begin{proof}
	From the positive semidefiniteness of $A$, the spectra of $A$ and $-A$ are disjoint. Thus, from Theorem~VII.2.1 of \cite{Bhatia}, the Sylvester equation $AX + XA = B$ has a unique solution.
	
	Let $X$ be the solution of the Sylvester equation.
	From $A^{\top} = A$, we have $A X^{\top} + X^{\top} A = (AX+XA)^{\top} = B^{\top} = B$, which means that $X^{\top}$ is also a solution of the Sylvester equation.
	Since the solution is unique, it implies $X^\top = X$.
	Also, from the positive semidefiniteness of $A$ and $AX + XA =B$, we have $X + A^{-1}XA = A^{-1}B$.
	Taking the trace and using ${\rm tr} (A^{-1}XA)={\rm tr} (X)$, we obtain ${\rm tr} X = {\rm tr} (A^{-1} B)/2$.
\end{proof}

\begin{theorem}\upshape
	\label{theorem_moment}
	For the elliptically symmetric deformation model \eqref{eq:am61}, the Wasserstein score functions are quadratic in $\bm{x}$. 
	Specifically, the Wasserstein score function for $\mu_i$ is
	\begin{align*}
		S_{\mu_i}^{{W}} (\bm{x},\bm{\theta}) = x_i-\mu_i,
	\end{align*}
	and the Wasserstein score function for $\Lambda_{ij}$ is
	\begin{align*}
		S_{\Lambda_{ij}}^W(\bm{x},\bm{\theta}) = \frac{1}{2} \bm{x}^{\top} A \bm{x} + \bm{b}^{\top} \bm{x} - {\rm E}_{\bm{\theta}} \left[ \frac{1}{2} \bm{x}^{\top} A \bm{x} + \bm{b}^{\top} \bm{x} \right],
	\end{align*}
	where $A$ is the unique solution of the Sylvester equation $\Lambda^2 A + A \Lambda^2 = -\Lambda \bm{e}_i \bm{e}_j^{\top} - \bm{e}_i \bm{e}_j^{\top} \Lambda$ and $\bm{b}=-A \bm{\mu}$.
\end{theorem}
\begin{proof}
	We show that the above $S^W$'s satisfy the Poisson equation \eqref{eq:poisson} directly.
	
	First, we consider the mean parameter $\mu_i$.
	From \eqref{eq:am61},
	\begin{align*}
		\frac{\partial}{\partial \mu_i} \log p(\bm{x},\bm{\theta}) = -\frac{\partial}{\partial x_i} \log p(\bm{x},\bm{\theta}) = -\nabla_{\bm{x}} \log p(\bm{x},\bm{\theta})^{\top} \bm{e}_i.
	\end{align*}
	Also, from Lemma~\ref{lem:quad},
	\begin{align*}
		\nabla_{\bm{x}} (x_i-\mu_i) = \bm{e}_i, \quad \Delta_{\bm{x}} (x_i-\mu_i) = 0.
	\end{align*}
	Therefore,
	\begin{align*}
		\nabla_{\bm{x}} \log p ({\bm{x}}, {\bm{\theta}})^{\top} \nabla_{\bm{x}} (x_i-\mu_i) + \Delta_{\bm{x}} (x_i-\mu_i) +
		\frac{\partial}{\partial \mu_i} \log
		p({\bm{x}}, {\bm{\theta}}) = 0.
	\end{align*}
	Thus, the Wasserstein score function for the mean parameter $\mu_i$ is
	\begin{align*}
		S_{\mu_i}^{{W}} (\bm{x},\bm{\theta}) = x_i-\mu_i.
	\end{align*}
	
	Next, we consider the deformation parameter $\Lambda_{ij}$.
	Since
	\begin{align*}
		\frac{\partial}{\partial \Lambda_{ij}} \| \Lambda (\bm{x}-\bm{\mu}) \| &= \frac{1}{2} \| \Lambda (\bm{x}-\bm{\mu}) \|^{-1} \frac{\partial}{\partial \Lambda_{ij}} (\bm{x}-\bm{\mu})^{\top} \Lambda^2 (\bm{x}-\bm{\mu}) \\
		&= \frac{1}{2} \| \Lambda (\bm{x}-\bm{\mu}) \|^{-1} (\bm{x}-\bm{\mu})^{\top} \frac{\partial \Lambda^2}{\partial \Lambda_{ij}} (\bm{x}-\bm{\mu}) \\
		&= \frac{1}{2} \| \Lambda (\bm{x}-\bm{\mu}) \|^{-1} (\bm{x}-\bm{\mu})^{\top} (\Lambda \bm{e}_i \bm{e}_j^{\top} + \bm{e}_i \bm{e}_j^{\top} \Lambda) (\bm{x}-\bm{\mu}),
	\end{align*}
	we have
	\begin{align*}
		\frac{\partial}{\partial \Lambda_{ij}} \log p(\bm{x}, \theta) &= \frac{\partial}{\partial \Lambda_{ij}} \log \det \Lambda + \frac{\partial}{\partial \Lambda_{ij}} \log g (\| \Lambda (\bm{x}-\bm{\mu}) \|) \\
		&= (\Lambda^{-1})_{ij} + \frac{g' (\| \Lambda (\bm{x}-\bm{\mu}) \|)}{g (\| \Lambda (\bm{x}-\bm{\mu}) \|)} \frac{\partial}{\partial \Lambda_{ij}} \| \Lambda (\bm{x}-\bm{\mu}) \| \\
		&= (\Lambda^{-1})_{ij} + \frac{g' (\| \Lambda (\bm{x}-\bm{\mu}) \|)}{2 \| \Lambda (\bm{x}-\bm{\mu}) \| g (\| \Lambda (\bm{x}-\bm{\mu}) \|)} (\bm{x}-\bm{\mu})^{\top} (\Lambda \bm{e}_i \bm{e}_j^{\top} + \bm{e}_i \bm{e}_j^{\top} \Lambda) (\bm{x}-\bm{\mu}) \\
		&= - \frac{g' (\| \Lambda (\bm{x}-\bm{\mu}) \|)}{\| \Lambda (\bm{x}-\bm{\mu}) \| g (\| \Lambda (\bm{x}-\bm{\mu}) \|)} \left( -\frac{1}{2} \bm{x}^{\top} L \bm{x} + \bm{x}^{\top} L \bm{\mu} - \frac{1}{2} \bm{\mu}^{\top} L \bm{\mu} \right) + (\Lambda^{-1})_{ij},
	\end{align*}
	where $L=\Lambda \bm{e}_i \bm{e}_j^{\top} + \bm{e}_i \bm{e}_j^{\top} \Lambda$.
	Let 
	\begin{align*}
		S(\bm{x}) = \frac{1}{2} \bm{x}^{\top} A \bm{x} + \bm{b}^{\top} \bm{x} - {\rm E}_{\theta} \left[ \frac{1}{2} \bm{x}^{\top} A \bm{x} + \bm{b}^{\top} \bm{x} \right],
	\end{align*}
	where $A$ is the unique solution of the Sylvester equation $\Lambda^2 A + A \Lambda^2 = -L$ and $\bm{b}=-A \bm{\mu}$.
	From Lemma~\ref{lem:quad} and Lemma~\ref{lem:sylvester},
	\begin{align*}
		\Delta_{\bm{x}} S(\bm{x}) = {\rm tr} A = -\frac{1}{2} {\rm tr} (\Lambda^{-2} L) = -\frac{1}{2} {\rm tr} (\Lambda^{-2} (\Lambda \bm{e}_i \bm{e}_j^{\top} + \bm{e}_i \bm{e}_i^{\top} \Lambda)) = -(\Lambda^{-1})_{ij}.
	\end{align*}
	Also,
	\begin{align*}
		&\nabla_{\bm{x}} \log p ({\bm{x}}, {\bm{\theta}})^{\top} \nabla_{\bm{x}} S(\bm{x}) \\ 
		=& \frac{g' (\| \Lambda (\bm{x}-\bm{\mu}) \|)}{g (\| \Lambda (\bm{x}-\bm{\mu}) \|)} \nabla_{\bm{x}} (\| \Lambda (\bm{x}-\bm{\mu}) \|)^{\top} (A\bm{x}+\bm{b}) \\
		=& \frac{g' (\| \Lambda (\bm{x}-\bm{\mu}) \|)}{\| \Lambda (\bm{x}-\bm{\mu}) \| g (\| \Lambda (\bm{x}-\bm{\mu}) \|)} (\Lambda^2 (\bm{x}-\bm{\mu}))^{\top} (A\bm{x}+\bm{b}) \\
		=& \frac{g' (\| \Lambda (\bm{x}-\bm{\mu}) \|)}{\| \Lambda (\bm{x}-\bm{\mu}) \| g (\| \Lambda (\bm{x}-\bm{\mu}) \|)} \left( \frac{1}{2} \bm{x}^{\top} (\Lambda^2 A + A \Lambda^2) \bm{x} + \bm{x}^{\top} (\Lambda^2 \bm{b}-A \Lambda^2 \bm{\mu}) -\bm{\mu}^{\top} \Lambda^2 \bm{b} \right) \\
		=& \frac{g' (\| \Lambda (\bm{x}-\bm{\mu}) \|)}{\| \Lambda (\bm{x}-\bm{\mu}) \| g (\| \Lambda (\bm{x}-\bm{\mu}) \|)} \left( -\frac{1}{2} \bm{x}^{\top} L \bm{x} + \bm{x}^{\top} L \bm{\mu} - \frac{1}{2} \bm{\mu}^{\top} L \bm{\mu} \right),
	\end{align*}
	where we used
	\begin{align*}
		\nabla_{\bm{x}} (\| \Lambda (\bm{x}-\bm{\mu}) \|) &= \frac{1}{2 \| \Lambda (\bm{x}-\bm{\mu}) \|} \nabla_{\bm{x}} (\| \Lambda (\bm{x}-\bm{\mu}) \|^2) \\
		&= \frac{1}{2 \| \Lambda (\bm{x}-\bm{\mu}) \|} \nabla_{\bm{x}} ((\bm{x}-\bm{\mu})^{\top} \Lambda^2 (\bm{x}-\bm{\mu})) \\
		&= \frac{1}{\| \Lambda (\bm{x}-\bm{\mu}) \|} \Lambda^2 (\bm{x}-\bm{\mu}).
	\end{align*}
	Therefore,
	\begin{align*}
		\nabla_{\bm{x}} \log p ({\bm{x}}, {\bm{\theta}})^{\top} \nabla_{\bm{x}} S(\bm{x}) + \Delta_{\bm{x}} S(\bm{x}) +
		\frac{\partial}{\partial \Lambda_{ij}} \log
		p({\bm{x}}, {\bm{\theta}}) = 0,
	\end{align*}
	which means that $S(\bm{x})$ is the Wasserstein score function for $\Lambda_{ij}$.
	
\end{proof}


Now, we consider the Wasserstein estimator $\hat{\bm{\theta}}_{W}$ defined as the zero of the Wasserstein score function \cite{LZ2019}.

\begin{corollary}\upshape \label{cor_ellip}
	Suppose that we have $n$ independent observations $\bm{x}_1,\dots,\bm{x}_n$ from the elliptically symmetric deformation model \eqref{eq:am61} where $n \geq d$.
	Then, the Wasserstein estimator $\hat{\bm{\theta}}_W=(\hat{\bm{\mu}}_W,\hat{\Lambda}_W)$ is the second-order moment estimator given by
	\begin{align*}
		\hat{\bm{\mu}}_W = \frac{1}{n} \sum_{t=1}^n \bm{x}_t, \quad \hat{\Lambda}_W = \left( \frac{1}{n} \sum_{t=1}^n (\bm{x}_t-\hat{\bm{\mu}}_W) (\bm{x}_t-\hat{\bm{\mu}}_W)^{\top} \right)^{-1/2},
	\end{align*}
	irrespective of the waveform $f(\bm{z})=g(\| \bm{z} \|)$.  
\end{corollary}
\begin{proof}
	From Theorem~\ref{theorem_moment}, the Wasserstein estimator is the solution of
	\begin{align}
		\frac{1}{n} \sum_{t=1}^n S_{\mu_i}^{{W}} (\bm{x}_t,\bm{\theta}) = \frac{1}{n} \sum_{t=1}^n  ((\bm{x}_t)_i-\mu_i) = 0 \label{est_mu}
	\end{align}
	for $i=1,\dots,d$ and
	\begin{align}
		\frac{1}{n} \sum_{t=1}^n S_{\Lambda_{ij}}^W(\bm{x}_t,\bm{\theta}) = \frac{1}{n} \sum_{t=1}^n \left( \frac{1}{2} \bm{x}_t^{\top} A \bm{x}_t + \bm{b}^{\top} \bm{x}_t \right) - {\rm E}_{\bm{\theta}} \left[ \frac{1}{2} \bm{x}^{\top} A \bm{x} + \bm{b}^{\top} \bm{x} \right] = 0 \label{est_lambda}
	\end{align}
	for $i,j=1,\dots,d$, where $A$ is the unique solution of the Sylvester equation $\Lambda^2 A + A \Lambda^2 = -\Lambda \bm{e}_i \bm{e}_j^{\top} - \bm{e}_i \bm{e}_j^{\top} \Lambda$ and $\bm{b}=-A \bm{\mu}$.
	(Note that the Wasserstein score function is a function from $\mathbb{R}^d$ to $\mathbb{R}$ and does not depend on $n$.)
	
	From \eqref{est_mu}, the Wasserstein estimator of $\bm{\mu}$ is 
	\begin{align*}
		\hat{\bm{\mu}}_W = \frac{1}{n} \sum_{t=1}^n \bm{x}_t.
	\end{align*}
	Also, since \eqref{est_lambda} implies that the second-order empirical moments match the second-order population moments and ${\rm Cov} [\bm{x}] = \Lambda^{-2}$ from \eqref{moment}, the Wasserstein estimator of $\Lambda$ is
	\begin{align*}
		\hat{\Lambda}_W = \left( \frac{1}{n} \sum_{t=1}^n (\bm{x}_t-\hat{\bm{\mu}}_W) (\bm{x}_t-\hat{\bm{\mu}}_W)^{\top} \right)^{-1/2}.
	\end{align*}
\end{proof}

\begin{theorem}\upshape
	When $f$ is Gaussian, the $F$-estimator $\hat{\bm{\theta}}_F$ and the $W$-estimator $\hat{\bm{\theta}}_W$ are identical.
\end{theorem}
\begin{proof}
	It is well known that the $F$-estimator (maximum likelihood estimator) for the Gaussian model is given by the second-order moment estimator.
\end{proof}

Note that \cite{Gelbrich} showed that the $L^2$-Wasserstein divergence for the elliptically symmetric deformation model \eqref{eq:am61} does not depend on the waveform, and is given by using the Bures--Wasserstein divergence between positive definite matrices \cite{Bhatia19}:
\begin{align*}
	D_W(p(\bm{x},\bm{\theta}_1),p(\bm{x},\bm{\theta}_2)) = \| \bm{\mu}_1 - \bm{\mu}_2 \|^2 + {\rm tr} (\Lambda_1^{-2} + \Lambda_2^{-2} -  2 (\Lambda_1^{-1} \Lambda_2^{-2} \Lambda_1^{-1})^{1/2}).
\end{align*}
It is an interesting future problem to derive the Wasserstein score function and Wasserstein estimator for general affine deformation models.

Regarding the geometric structure of the elliptically symmetric deformation model \eqref{eq:am61}, we obtain the following.
See Figure~\ref{fig1}.

\begin{theorem}\upshape
	When $f$ is spherically symmetric, the model $\mathcal{M}_f$ is orthogonal in $L^2$ to ${\mathcal{F}}_S$ at the origin ${\bm{\mu}}=0, \Lambda=I$ of $\mathcal{M}_f$ with respect to the Wasserstein metric.
\end{theorem}
\begin{proof}\upshape
	Let $\delta p({\bm{x}})$ be a tangent vector of ${\mathcal{F}}_S$ at the origin.
	Since all $p({\bm{x}})$ in  ${\mathcal{F}}_S$ satisfy the standard conditions (\ref{eq:am2}), (\ref{eq:am3}), and (\ref{eq:am4}), $\delta p({\bm{x}})$ is orthogonal to any quadratic function of ${\bm{x}}$.  
	Since the $W$-score functions of $\mathcal{M}_f$ are quadratic functions from Theorem~\ref{theorem_moment}, it implies that $\delta p({\bm{x}})$ is orthogonal to the basis functions of the tangent space of $\mathcal{M}_f$, with respect to the Wasserstein metric.
\end{proof}

Here, we discuss the relation between the current Wasserstein estimator $\hat{{\bm \theta}}_{W}$ and the estimator $\hat{{\bm \theta}}_{Wp}$ in \eqref{am_est} defined as the projection of the empirical distribution with respect to the Wasserstein distance.
For the one-dimensional location-scale model, \cite{AM2022} studied the estimator $\hat{{\bm \theta}}_{Wp}$ in \eqref{am_est} by using the order statistics $x_{(i)}$.  
This estimator minimizes the Wasserstein distance between the empirical distribution and the model.
Here, we show that this estimator $\hat{\bm{\theta}}_{Wp}=(\hat{\mu}_{Wp},\hat{\sigma}_{Wp})$ is asymptotically equivalent to the Wasserstein estimator $\hat{{\bm \theta}}_{W}$, which coincides with the second-order moment estimator from Theorem~\ref{theorem_moment}.  
We assume $\mu=0$ without loss of generality.  
The estimator \eqref{am_est} of the location is 
\begin{equation*}
	\hat{\mu}_{Wp} = \frac 1n \sum_{i=1}^n x_{(i)} =
	\frac 1n \sum_{i=1}^n x_i,
\end{equation*}
which is the empirical mean and coincides with the moment estimator.  
Also, the estimator \eqref{am_est} of the scale is
\begin{equation*}
	\hat{\sigma}_{Wp} = \sum_{i=1}^n k_i x_{(i)},
\end{equation*}
where
\begin{equation*}
	k_i = \int^{z_i}_{z_{i-1}} zf(z) dz.
\end{equation*}
Here, $z_i$ is the $i$-th equipartition point of $f(z)$ defined by
\begin{equation*}
	z_i = F^{-1} \left( \frac{i}{n} \right),
\end{equation*}
where $F$ is the cumulative distribution function of $f(z)$.  
From $\mu=0$, we have $x_{(i)} = \sigma z_i + O_p(n^{-1/2})$ asymptotically \citep{David}. 
Hence, 
\begin{equation*}
	k_i \approx \frac 1n z_i \approx \frac 1n \frac{x_{(i)}}{\sigma},
\end{equation*}
which leads to
\begin{equation*}
	\hat{\sigma}_{Wp} = \sum_{i=1}^n k_i x_{(i)} \approx \frac{1}{n \sigma} \sum_{i=1}^n x^2_{(i)} = \frac{1}{n \sigma} \sum_{i=1}^n x^2_{i}.
\end{equation*}
Since $\hat{\sigma}_{Wp} \approx \sigma$ asymptotically,
\begin{equation*}
	\hat{\sigma}_{Wp}^2 \approx \frac{1}{n} \sum_{i=1}^n x^2_{i}.
\end{equation*}
This shows that $\hat{\bm{\theta}}_{Wp}=(\hat{{\bm \mu}}_{Wp},\hat{\sigma}_{Wp})$ asymptotically coincides with the second-order moment estimator $\hat{{\bm \theta}}_{W}=(\hat{{\bm \mu}}_{W},\hat{\sigma}_{W})$.

\section{Wasserstein covariance and robustness}\label{sec:robust}
Following \cite{LZ2019}, we define the Wasserstein covariance ($W$-covariance) matrix ${\rm Var}_{\theta}^{\mathrm{W}} [\hat{\bm{\theta}}]$ of an estimator $\hat{\bm{\theta}}$ by the positive semidefinite matrix given by
\begin{align}
	{\rm Var}_{\theta}^{\mathrm{W}} [\hat{\bm{\theta}}] = ({\rm E}_{\theta} [ (\nabla_{\bm{x}} \hat{{\theta}}_a)^{\top} (\nabla_{\bm{x}} \hat{{\theta}}_b) ])_{ab}, \label{Wcov}
\end{align}
where $\nabla_{\bm{x}} \hat{\bm{\theta}}$ is assumed to be square integrable.
\citet{LZ2019} showed the Wasserstein--Cramer--Rao inequality
\begin{equation}
	\label{eq:WCR}
	{\rm Var}_{\theta}^{\mathrm{W}} (\hat{\bm{\theta}}) \succeq \left( \frac{\partial}{\partial \theta} {\rm E}_{\theta} [\hat{\bm{\theta}}] \right)^{\top} G_{\mathrm{W}} (\theta)^{-1} \left( \frac{\partial}{\partial \theta} {\rm E}_{\theta} [\hat{\bm{\theta}}] \right),
\end{equation}  
where
\begin{align*}
	\frac{\partial}{\partial \theta} {\rm E}_{\theta} [\hat{\bm{\theta}}] := \left( \frac{\partial}{\partial \theta_b} {\rm E}_{\theta} [ \hat{\bm{\theta}}_a ] \right)_{ab}.
\end{align*}
A consistent estimator $\hat{\bm{\theta}}$ is said to be Wasserstein efficient ($W$-efficient) if its Wasserstein covariance asymptotically satisfies \eqref{eq:WCR} with equality.   
We give a proof of the Wasserstein--Cramer--Rao inequality based on the Cauchy--Schwarz inequality in the Appendix.

We show that the Wasserstein covariance of an estimator can be viewed as a measure of robustness against additive noise.  
Suppose that $X \sim p ({\bm{x}}, {\bm{\theta}})$ and we estimate $\bm{\theta}$ from the noisy observation $\widetilde{X}=X+Z$, where $Z$ is independent from $X$, ${\rm E} [Z] = 0$ and ${\rm Var} [Z] = \sigma^2 I$ with $\sigma^2$ sufficiently small.
The probability density of $\widetilde{X}$ is given by $\widetilde{p}({\bm{x}}, {\bm{\theta}})=p ({\bm{x}}, {\bm{\theta}}) \ast q(\bm{x})$, where $\ast$ is the convolution and $q$ is the probability density of the noise $Z$. 
Namely, the noise changes the waveform from $p$ to $\widetilde{p}$. 
Generally, the estimator degrades when the noise is added. 
Here, we quantify the robustness of an estimator against noise based on how much its variance increases due to noise. 
Namely, we focus on ${{\rm Var}_{\theta} [\hat{\bm{\theta}} (X+Z)] - {\rm Var}_{\theta} [\hat{\bm{\theta}}(X)]}$, which does not depend on the exact distribution of $Z$ but only its first two moments.
If this quantity is small, it implies that the estimator is not much affected by noise contamination, which can be viewed as its robustness.
This quantity is closely related to the Wasserstein covariance as follows.

\begin{theorem}\label{th_cov}
	The Wasserstein covariance satisfies
	\begin{align*}
		{\rm Var}^{\mathrm{W}}_{\theta} [\hat{\bm{\theta}}]_{ab} =& \lim_{\sigma^2 \to 0} \frac{{\rm Var}_{\theta} [\hat{\bm{\theta}} (X+Z)]_{ab} - {\rm Var}_{\theta} [\hat{\bm{\theta}}(X)]_{ab}}{\sigma^2} \\
		& \quad - \frac{1}{2} \left( {\rm Cov}_{\theta} [ \hat{{\theta}}_a (X), \Delta \hat{{\theta}}_b (X) ] + {\rm Cov}_{\theta} [ \hat{{\theta}}_b (X), \Delta \hat{{\theta}}_a (X) ] \right),
	\end{align*}
	where $\Delta$ is the Laplacian.
	In particular, if $\Delta \hat{\bm{\theta}}_a (X)$ is constant for every $a$ (e.g. $\hat{\bm{\theta}}$ is quadratic in $\bm{x}$), then
	\begin{equation*}
		{\rm Var}^{\mathrm{W}}_{\theta} [\hat{\bm{\theta}}] = \lim_{\sigma^2 \to 0} \frac{{\rm Var}_{\theta} [\hat{\bm{\theta}} (X+Z)] - {\rm Var}_{\theta} [\hat{\bm{\theta}}(X)]}{\sigma^2}.
	\end{equation*}
\end{theorem}
\begin{proof}
	By Taylor expansion, for sufficiently small $\bm{z}$,
	\begin{align*}
		\hat{{\theta}}_a (\bm{x}+\bm{z}) \approx \hat{{\theta}}_a(\bm{x}) + \sum_i \frac{\partial \hat{{\theta}}_a}{\partial x_i}(\bm{x}) z_i + \frac{1}{2} \sum_{i,j} \frac{\partial^2 \hat{{\theta}}_a}{\partial x_i \partial x_j}(\bm{x}) z_i z_j.
	\end{align*}
	From ${\rm E} [Z] = 0$, ${\rm Var} [Z] = \sigma^2 I$ and the independence of $X$ and $Z$,
	\begin{align*}
		&{\rm E}_{\theta} [\hat{{\theta}}_a (X+Z)] \\
		=& {\rm E}_{\theta} [\hat{{\theta}}_a(X)] + \sum_i {\rm E}_{\theta} \left[ \frac{\partial \hat{{\theta}}_a}{\partial x_i} (X) \right] {\rm E}[z_i] + \frac{1}{2} \sum_{i,j} {\rm E}_{\theta} \left[ \frac{\partial^2 \hat{{\theta}}_a}{\partial x_i \partial x_j} (X) \right] {\rm E} [z_i z_j] + o(\sigma^2) \\
		=& {\rm E}_{\theta} [\hat{{\theta}}_a(X)] + \frac{1}{2} {\rm E}_{\theta} [ \Delta\hat{{\theta}}_a (X) ] \sigma^2 + o(\sigma^2). 
	\end{align*}
	Also,
	\begin{align*}
		&{\rm E}_{\theta} [\hat{{\theta}}_a (X+Z) \hat{{\theta}}_b (X+Z)] \\
		=& {\rm E}_{\theta} [\hat{{\theta}}_a(X) \hat{{\theta}}_b(X)] + \frac{1}{2} {\rm E}_{\theta} [ \hat{{\theta}}_a(X) \Delta \hat{{\theta}}_b (X) + \hat{{\theta}}_b(X) \Delta \hat{{\theta}}_a (X) ] \sigma^2 + {\rm Var}_{\theta}^{\mathrm{W}} [\hat{\bm{\theta}}]_{ab} \sigma^2 + o(\sigma^2).
	\end{align*}
	Then,
	\begin{align*}
		&{\rm Var}_{\theta} [\hat{\bm{\theta}} (X+Z)]_{ab} \\
		=& {\rm E}_{\theta} [\hat{{\theta}}_a (X+Z) \hat{{\theta}}_b (X+Z)] - {\rm E}_{\theta} [\hat{{\theta}}_a (X+Z)] {\rm E}_{\theta} [\hat{{\theta}}_b (X+Z)] \\
		=& {\rm Var}_{\theta} [\hat{\bm{\theta}}(X)]_{ab} + {\rm Var}^{\mathrm{W}}_{\theta} [\hat{\bm{\theta}}]_{ab} \sigma^2 \\
		& \qquad + \frac{1}{2} \left( {\rm Cov}_{\theta} [ \hat{{\theta}}_a (X), \Delta \hat{{\theta}}_b (X) ] + {\rm Cov}_{\theta} [ \hat{{\theta}}_b (X), \Delta \hat{{\theta}}_a (X) ] \right) \sigma^2 + o(\sigma^2),
	\end{align*}
	where the covariance term vanishes when $\hat{\bm{\theta}}$ is quadratic in $\bm{x}$.
\end{proof}

For the elliptically symmetric deformation model \eqref{eq:am61}, from Theorem~\ref{th_cov} and Corollary~\ref{cor_ellip}, the Wasserstein covariance quantifies the robustness of the Wasserstein estimator and the Wasserstein--Cramer--Rao inequality gives its limit.
It is an interesting future problem to investigate when the Wasserstein estimator attains the Wasserstein efficiency.
Note that the Fisher efficiency (in finite samples), which is defined by the usual Cramer--Rao inequality, is attained by the maximum likelihood estimator if and only if the estimand is the expectation parameter of an exponential family.

\section{Conclusion}
\label{sec:conc}

Statistical inference based on the likelihood principle enjoyed great successes, and information geometry has played an important role in it.  However, Wasserstein divergence gives another viewpoint, which is based on the geometric structure of the sample space $X$.  There are many applications of the Wasserstein geometry not only to the transportation problem but to computer vision and image analysis and recently deep learning in AI.

We studied the Wasserstein statistics  using the framework of \cite{LZ2019}, proving that the Wasserstein covariance  quantifies robustness against the convolutional waveform deformation due to observation noise.  
We further studied $W$-statistics of the affine deformation model.
We showed $F$-efficiency and $W$-efficiency of estimators $\hat{\bm{\theta}}_F$ and $\hat{\bm{\theta}}_W$.   We elucidated how the waveform $f$ contributes to the efficiencies.  The Gaussian distribution gives the only waveform in which the $F$-estimator and $W$-estimator coincide.

Estimation of a covariance matrix, especially in high dimensions under special structure (e.g. low-rankness, sparsity), has been an important problem in statistics. 
It is an interesting future problem to investigate whether the Wasserstein geometry is helpful in covariance estimation.

Other than the elliptically symmetric deformation model, it is difficult in general to derive the Wasserstein score, which corresponds to the infinitesimal optimal transport. It is an interesting future problem to explore other statistical models for which the Wasserstein score is obtained in closed form. Also, it would be useful to develop approximation techniques for the Wasserstein score.

The present paper is only a first step to construct general Wasserstein statistics.  In future work, we need to use more general statistical models. 
We also need to extend our approach to statistical theories of hypothesis testing, pattern classification, clustering and many other statistical problems based on the Wasserstein geometry.

\section*{Acknowledgements}
We thank the referees for constructive comments.
We thank Asuka Takatsu and Tomonari Sei for helpful comments.
We thank Emi Namioka for drawing the figures.
Takeru Matsuda was supported by JSPS KAKENHI Grant Numbers 19K20220, 21H05205, 22K17865 and JST Moonshot Grant Number JPMJMS2024.

\appendix

	\section{Proof of Wasserstein--Cramer--Rao inequality}
	For random vectors $U$ and $V$, 
	\begin{align*}
		0 \leq {\rm E} \| tU+V \|^2 = {\rm E} [\| U \|^2] t^2 + 2 {\rm E} [U^{\top} V] t + {\rm E} [\| V \|^2]
	\end{align*}
	for every $t$.
	Thus, by considering the discriminant of the quadratic equation,
	\begin{align}
		{\rm E} [U^{\top} V]^2 \leq {\rm E} [\| U \|^2] {\rm E} [\| V \|^2]. \label{CS}
	\end{align}
	Substituting 
	\begin{align*}
		U = \sum_i a_i \nabla_{\bm{x}} \hat{\theta}_i, \quad V = \sum_j b_j \nabla_{\bm{x}} S_j^W
	\end{align*}
	into \eqref{CS} yields
	\begin{align}
		\left( \sum_{i,j} a_i b_j {\rm E}_{\theta} [(\nabla_{\bm{x}} \hat{\theta}_i)^{\top} (\nabla_{\bm{x}} S_j^W)] \right)^2 \leq {\rm E}_{\theta} \left[ \left\| \sum_i a_i \nabla_{\bm{x}} \hat{\theta}_i \right\|^2 \right] {\rm E}_{\theta} \left[ \left\| \sum_j b_j \nabla_{\bm{x}} S_j^W \right\|^2 \right]. \label{CS2}
	\end{align}
	
	From the property \eqref{eq:poisson} of the Wasserstein score function,
	\begin{align*}
		\frac{\partial}{\partial \theta_j} {\rm E}_{\theta} [\hat{\theta}_i] &= \int \hat{\theta}_i(x) \frac{\partial}{\partial \theta_j} p(x, \theta) {\rm d} x \\
		&= -\int \hat{\theta}_i(x) \nabla_{\bm{x}} \cdot (p(x, \theta) \nabla_{\bm{x}} S_j^W(x, \theta)) {\rm d} x \\
		&= -\int (\nabla_{\bm{x}} \cdot (\hat{\theta}_i(x) p(x, \theta) \nabla_{\bm{x}} S_j^W(x, \theta)) - (\nabla_{\bm{x}} \hat{\theta}_i(x))^{\top} (\nabla_{\bm{x}} S_j^W(x, \theta)) p(x, \theta) ) {\rm d} x \\
		&= {\rm E}_{\theta} [(\nabla_{\bm{x}} \hat{\theta}_i)^{\top} (\nabla_{\bm{x}} S_j^W)],
	\end{align*}
	which yields
	\begin{align}
		\sum_{i,j} a_i b_j {\rm E}_{\theta} [(\nabla_{\bm{x}} \hat{\theta}_i)^{\top} (\nabla_{\bm{x}} S_j^W)] = \sum_{i,j} a_i b_j \frac{\partial}{\partial \theta_j} {\rm E}_{\theta} [\hat{\theta}_i] = \bm{a}^{\top} \left( \frac{\partial}{\partial \theta} {\rm E}_{\theta} [\hat{\theta}] \right) \bm{b}. \label{App1}
	\end{align}
	From the definition of the Wasserstein covariance \eqref{Wcov},
	\begin{align}
		{\rm E}_{\theta} \left[ \left\| \sum_i a_i \nabla_{\bm{x}} \hat{\theta}_i \right\|^2 \right] &= {\rm E}_{\theta} \left[ \sum_{i,j} a_i a_j (\nabla_{\bm{x}} \hat{\theta}_i)^{\top} (\nabla_{\bm{x}} \hat{\theta}_j) \right] \nonumber \\
		&= \sum_{i,j} a_i a_j {\rm Var}_{\theta}^{\mathrm{W}} (\hat{\theta})_{ij} \nonumber \\
		&= \bm{a}^{\top} {\rm Var}_{\theta}^{\mathrm{W}} (\hat{\theta}) \bm{a}. \label{App2}
	\end{align}
	From the definition of the Wasserstein information matrix \eqref{WIM},
	\begin{align}
		{\rm E}_{\theta} \left[ \left\| \sum_j b_j \nabla_{\bm{x}} S_j^W \right\|^2 \right] &= {\rm E}_{\theta} \left[ \sum_{i,j} b_i b_j (\nabla_{\bm{x}} S_i^W)^{\top} (\nabla_{\bm{x}} S_j^W) \right] \nonumber \\ 
		&= \sum_{i,j} b_i b_j G_{\mathrm{W}} (\theta)_{ij} \nonumber \\
		&= \bm{b}^{\top} G_{\mathrm{W}} (\theta) \bm{b}. \label{App3}
	\end{align}
	
	Substituting \eqref{App1}, \eqref{App2} and \eqref{App3} into \eqref{CS2}, we obtain
	\begin{align*}
		\left( \bm{a}^{\top} \left( \frac{\partial}{\partial \theta} {\rm E}_{\theta} [\hat{\theta}] \right) \bm{b} \right)^2 \leq \bm{a}^{\top} {\rm Var}_{\theta}^{\mathrm{W}} (\hat{\theta}) \bm{a} \cdot \bm{b}^{\top} G_{\mathrm{W}} (\theta) \bm{b}.
	\end{align*}
	By putting 
	\begin{align*}
		\bm{b}=G_{\mathrm{W}} (\theta)^{-1} \left( \frac{\partial}{\partial \theta} {\rm E}_{\theta} [\hat{\theta}] \right) \bm{a},
	\end{align*}
	it leads to 
	\begin{align*}
		\bm{a}^{\top} \left( \frac{\partial}{\partial \theta} {\rm E}_{\theta} [\hat{\theta}] \right) G_{\mathrm{W}} (\theta)^{-1} \left( \frac{\partial}{\partial \theta} {\rm E}_{\theta} [\hat{\theta}] \right) \bm{a} \leq \bm{a}^{\top} {\rm Var}_{\theta}^{\mathrm{W}} (\hat{\theta}) \bm{a},
	\end{align*}
	which is equal to the Wasserstein Cramer--Rao inequality \eqref{eq:WCR}.

\end{document}